\documentclass[12pt, oneside]{amsart}

\usepackage[utf8]{inputenc}
\usepackage[T1]{fontenc}

\usepackage{amssymb}
\usepackage{amsthm}
\usepackage{amsmath}
\usepackage{mathrsfs}    %%% PARA PSEUDOGRUPOS
\usepackage[new]{old-arrows}   %%% PARA SETAS DE INJEÇÕES
\usepackage{accents} %%% PARA ACENTOS FIRULENTOS
\usepackage{tikz-cd}
\tikzcdset{arrow style=math font}

\usepackage[a4paper,left=2cm,top=3cm,right=2cm,bottom=2cm]{geometry}

\usepackage[all,cmtip]{xy}

\usepackage{enumerate}

\usepackage{indentfirst}

\usepackage{hyperref}
\hypersetup{colorlinks=true,allcolors=black}  %%% COR DO LINK

\usepackage{pstricks}
\usepackage{graphicx}

\usepackage{multicol}

\DeclareMathOperator{\diam}{diam}
\DeclareMathOperator{\codim}{codim}

\DeclareMathOperator{\stalk}{stalk}

\DeclareMathOperator{\vol}{vol}

\DeclareMathOperator{\susp}{S}

\mathchardef\ordinarycolon\mathcode`\:
\mathcode`\:=\string"8000
\begingroup \catcode`\:=\active
  \gdef:{\mathrel{\mathop\ordinarycolon}}
\endgroup

\begin{document}

\title{Transverse sphere theorems for Riemannian foliations}

\author{Francisco C.~Caramello Jr.}
\address{Departamento de Matemática, Universidade Federal de Santa Catarina, R. Eng. Agr. Andrei Cristian Ferreira, 88040-900, Florianópolis - SC, Brazil}
\email{francisco.caramello@ufsc.br}

\author{Francisco A. Neubauer}
\address{Departamento de Matemática, Universidade Federal de Santa Catarina, R. Eng. Agr. Andrei Cristian Ferreira, 88040-900, Florianópolis - SC, Brazil}
\email{francisco.neubauer@posgrad.ufsc.br}


%%% AMBIENTES %%%
\newenvironment{proofoutline}{\proof[Proof outline]}{\endproof}
\theoremstyle{definition}
\newtheorem{example}{Example}[section]
\newtheorem{definition}[example]{Definition}
\newtheorem{remark}[example]{Remark}
\theoremstyle{plain}
\newtheorem{proposition}[example]{Proposition}
\newtheorem{theorem}[example]{Theorem}
\newtheorem{lemma}[example]{Lemma}
\newtheorem{corollary}[example]{Corollary}
\newtheorem{claim}[example]{Claim}
\newtheorem{conjecture}[example]{Conjecture}
\newtheorem{thmx}{Theorem}
\renewcommand{\thethmx}{\Alph{thmx}} % "letter-numbered" theorems
\newtheorem{corx}{Corollary}
\renewcommand{\thecorx}{\Alph{corx}} % "letter-numbered" corollaries

%%% MACROS %%%
\newcommand{\dif}[0]{\mathrm{d}}
\newcommand{\od}[2]{\frac{d #1}{d #2}}
\newcommand{\pd}[2]{\frac{\partial #1}{\partial #2}}
\newcommand{\dcov}[2]{\frac{\nabla #1}{d #2}}
\newcommand{\proin}[2]{\left\langle #1, #2 \right\rangle}
\newcommand{\f}[0]{\mathcal{F}}
\newcommand{\g}[0]{\mathcal{G}}
\newcommand{\metric}{\ensuremath{\mathrm{g}}}
\newcommand{\tmetric}{\ensuremath{\mathrm{g}_\intercal}}
\newcommand{\dslash}[2]{#1 {/}\mkern-4mu{/} #2}

\begin{abstract}
We extend the classical theory of sphere theorems to the transverse geometry of Riemannian foliations. In this setting, we establish transverse analogues of the Grove–Shiohama diameter sphere theorem and of the Berger–Klingenberg quarter-pinched sphere theorem. First, we prove that if a Killing foliation of a compact, connected manifold has transverse sectional curvature greater than 1 and transverse diameter greater than $\pi/2$, then, after an arbitrarily small deformation, the resulting foliation has leaf space homeomorphic to a good spherical orbifold. Moreover, the space of leaf closures of the original foliation is realized as a further quotient of this spherical model by a torus action. Using this deformation theory we also prove that the space of leaf closures of a Killing foliation of a compact manifold is the Gromov--Hausdorff limit of a sequence of orbifolds. Under transverse quarter-pinching, we prove that this convergence is non-collapsing. As a consequence, we obtain that a complete Riemannian foliation with quarter-pinched transverse sectional curvature and codimension at least 3 develops on the universal cover to a simple foliation given by the fibers of a submersion onto the sphere.
\end{abstract}

\maketitle
\setcounter{tocdepth}{1}
\tableofcontents

\section{Introduction}
In essence, a Riemannian foliation is a foliation with locally equidistant leaves. These objects occur in many contexts, the more elementary one being the case of Riemannian submersions, whose fibers define \textit{simple} Riemannian foliations. In this spirit, incidentally, the study of the transverse geometry of more general Riemannian foliations allows one to extend the domains of classical Riemannian geometry to their rather singular leaf spaces. Another example is that of foliations given by the orbits of isometric actions, the so-called \emph{homogeneous} Riemannian foliations. Properties of this particular case also generalize; for instance, there is a robust structural theory for complete Riemannian foliations, due mainly to P.~Molino \cite{molino}, that establishes that the leaf closures of such a foliation ${\mathcal{F}}$ form another Riemannian foliation $\overline{{\mathcal{F}}}$, this one possibly singular. The dynamics of ${\mathcal{F}}$, therefore, is relatively constrained. It is, indeed, described by the action of a locally constant sheaf $\mathscr{C}_{\mathcal{F}}$ of Lie algebras of germs of transverse Killing vector fields (see Section \ref{section preliminaries} for more details). When this sheaf is globally constant, ${\mathcal{F}}$ is called a \emph{Killing} foliation, following \cite{mozgawa}. This class encompasses complete homogeneous Riemannian foliations, as well as complete Riemannian foliations on simply connected manifolds. For this reason, Killing foliations are particularly relevant in the study of developability, i.e., the question of whether it is possible to lift a given Riemannian foliation $(M,\f)$ to some cover of $M$ so that the lifted foliation is simple --- or, in other words, so that the holonomy of the foliation gets trivialized after the lift.

In this paper we are interested in obtaining transverse versions of famous sphere rigidity theorems from classical Riemannian geometry. More precisely, let us first recall the Grove--Shiohama diameter theorem \cite{grove}, which asserts that a connected, complete Riemannian manifold $M$ with sectional curvature satisfying $\sec_M\geq 1$ and whose diameter verifies $\diam_M > \pi/2$ must be homeomorphic to a sphere. Perelman further generalized this result to Alexandrov spaces \cite{perelman}, which we apply here in order to prove the following.

\begin{thmx}[Theorem \ref{theorem transverse diameter sphere}]\label{theoremA}
Let $\f$ be a $q$-codimensional Killing foliation of a connected, compact manifold $M$, with transverse sectional curvature and transverse diameter satisfying $\sec_\f > 1$ and $\diam_\intercal(\f) > \pi/2$, respectively. Then $M$ admits a foliation $\g$, arbitrarily close to $\f$, such that $M/\g$ is homeomorphic to the quotient of $\mathbb{S}^q$ by a finite group $\Gamma<\mathrm{O}(q+1)$. Moreover there is a continuous action of a torus $\mathbb{T}^d$ on $\mathbb{S}^q/\Gamma$ such that
\[\frac{M}{\overline{\f}}\cong \frac{\mathbb{S}^q/\Gamma}{\mathbb{T}^d}.\]
\end{thmx}

Alternatively, one can also obtain $M/\overline{\f}$ as a quotient of $\mathbb{S}^q$ by an extension of $\mathbb{T}^d$ by $\Gamma$ (see Corollary \ref{corollary of theorem A}). As elucidated in Section \ref{section preliminaries}, the number $d$ is the dimension of the structural algebra of $\f$, which equals $\dim(\overline{\f})-\dim(\f)$. Notice that, when $\f$ is the trivial foliation of $M$ by points, Theorem \ref{theoremA} reduces to the classical Grove--Shiohama diameter theorem.

We prove Theorem \ref{theoremA} via a deformation technique for Killing foliations that allows one to approximate a Killing foliation of a compact manifold by \textit{closed} Riemannian foliations --- i.e., Riemannian foliations whose leaves are all closed submanifolds. It traces back to Ghys \cite{ghys}, who established it in the simply connected case, and was later generalized by Haefliger and Salem to a broader class of Killing foliations in \cite{haefliger2}. In \cite{caramello} the authors showed that, in a technical sense, the transverse geometry of $\f$ varies continuously throughout the deformation, which then allows one to apply results of the Riemannian geometry of orbifolds in the study of its transverse geometry. Additionally, some transverse algebraic topological invariants, such as the basic Euler characteristic, remain constant during the deformations (see also \cite{caramello2}). Here we further extend this deformation theory by showing that it produces a sequence $\g_i$ of closed Riemannian foliations converging to $\f$, in the $C^0$ topology (see Proposition \ref{proposition uniform convergence of foliations}). Moreover, we show that, a bundle-like metric for $\f$ deforms to bundle-like metrics $\metric_i$ for $\g_i$, so that $\metric_i\to\metric$ uniformly (see Lemma \ref{lemma: C0 convergence of metrics}). With this machinery, we establish that lower bounds in $\diam_\intercal(\f)$ are maintained under this type of deformation (see Corollary \ref{corollary diameters of closed approximations}), which leads to Theorem \ref{theoremA}. With the same technique we prove:

\begin{thmx}[Theorem \ref{theorem regular convergence implies GH convergence}]\label{theoremB}
Let $\f$ be a Killing foliation of a connected, compact manifold $M$. Then there exists a sequence of closed Riemannian foliations $\g_i$ of $M$ such that $M/\g_i\to M/\overline{\f}$ in the Gromov--Hausdorff sense.
\end{thmx}

Our next result can be seen as a transverse version of the celebrated Berger--Klingenberg quarter-pinched sphere theorem. Recall that this classical result asserts that a simply connected, complete Riemannian manifold $M$ with $1/4<\sec_M\leq 1$ is homeomorphic to a sphere. Some decades later, the differentiable version of the quarter-pinched sphere theorem was proved by Brendle and Schoen \cite{brendle} via Ricci flow methods. Incidentally, to obtain such improvement the authors relied on techniques from \cite{bohm} by Böhm and Wilking, which in turn establish that a complete, quarter-pinched orbifold $\mathcal{O}$ with $\dim\mathcal{O}\geq3$ is finitely developable: it is a quotient of a manifold by the isometric action of a finite group. By the differentiable quarter-pinched sphere theorem one then concludes that this manifold must be, furthermore, diffeomorphic to a sphere. In other words, the orbifold version of the quarter-pinched sphere theorem holds in dimensions greater than or equal to 3. Coupling it with Theorem \ref{theoremB} one proves that quarter-pinching prevents collapse in the convergence $M/\g_i\to M/\overline{\f}$, leading to:

\begin{thmx}[Theorem \ref{teoprincipal}]\label{theoremC}
Let $\f$ be a $q$-codimensional complete Riemannian foliation of a connected manifold $M$, with $1/4<\sec_\f< 1$ and $q\geq3$. Then $\f$ develops, on the universal covering $\hat{M}$, to a simple foliation $\hat{\f}$ given by the fibers of a submersion onto $\mathbb{S}^q$.
\end{thmx}

Our method of proof requires a strict pinching in $\sec_\f$, but we conjecture that the hypothesis $1/4<\sec_\f\leq 1$ is sufficient. It is worth noticing, for instance, that the conclusion in Theorem \ref{theoremC} holds when $\sec_\f\equiv 1$ (in this case, moreover, in any codimension). Indeed, if $\sec_\f\equiv 1$, then by the local version of the Cartan--Ambrose--Hicks Theorem one can assume that the local quotients $T_i$ of $\f$ (see Section \ref{section preliminaries}) are open sets of $\mathbb{S}^q$. Since the sphere is homogeneous, the corresponding holonomy transformations between them can be extended to global isometries. Therefore the foliation $\f$ is transversely homogeneous, in the sense of \cite{blumenthal}, and hence develops to a fibration over $\mathbb{S}^q$, by \cite[Theorem 1]{blumenthal}.

\section{Preliminaries}\label{section preliminaries}

We start this section by briefly defining the pertinent objects, mainly to fix our notation throughout the text. Then we recall some classical results on Riemannian foliations and establish Proposition \ref{proposition uniform convergence of foliations}, a thechnical tool to be used later. For simplicity, throughout the paper we work in the category of smooth manifolds and smooth maps, unless otherwise explicitly stated, and we will henceforth omit the adjective ``smooth''. The symbol $\cong$, however, will denote homeomorphisms (or isomorphisms, when applied to groups). A \textit{singular foliation} of an $n$-dimensional manifold $M$ is a partition ${\mathcal{F}}$ of $M$ into connected, immersed submanifolds, which we call \textit{leaves}, such that, for each $L\in{\mathcal{F}}$ and each $x\in L$, any given vector $v\in T_xL$ can be extended to a smooth vector field $V$ in the module $\mathfrak{X}({\mathcal{F}})$ of all vector fields on $M$ that take values in the distribution (of varying rank) $T{\mathcal{F}}$ formed by the tangent spaces of the leaves. The \textit{leaf space} $M/{\mathcal{F}}$ of ${\mathcal{F}}$ is the quotient of $M$ by the equivalence relation that identifies points on the same leaf. The set of all closures of leaves of $\f$ is $\overline{\f}:=\{\overline{L}\ |\ L\in\f\}$. When $\overline{\f}=\f$, we say that $\f$ is \textit{closed}. For each $x\in M$ we denote the leaf containing it by $\f(x)$. 

When all leaves have the same dimension $p$ we say that ${\mathcal{F}}$ is a \textit{regular foliation} of $M$, and that $q:=n-p$ is the \textit{codimension} of $\f$. From now on we will assume all foliations are regular, unless otherwise explicitly stated, and henceforth omit the adjective ``regular''. In this case, of course, $T{\mathcal{F}}$ is also a vector bundle over $M$, and we can consider the \textit{normal bundle} $\nu\f:=TM/T\f$ of $\f$. Moreover, $\f$ can also be defined in terms of an open cover $\{U_i\}_{i\in I}$ of $M$, submersions with connected fibers $\pi_i:U_i\to T_i$, where $T_i\subset\mathbb{R}^q$ are open subsets, and diffeomorphisms $\gamma_{ij}:\pi_j(U_i\cap U_j)\to\pi_i(U_i\cap U_j)$ such that $\gamma_{ij}\circ\pi_j|_{U_i\cap U_j}=\pi_i|_{U_i\cap U_j}$ for all $i,j\in I$. The collection $(U_i,\pi_i,\gamma_{ij})$ is a \textit{cocycle} representing $\f$. The  \textit{holonomy pseudogroup} of $\f$ associated to $\gamma:=\{\gamma_{ij}\}$ is the pseudogroup $\mathscr{H}_\gamma$ of local diffeomorphisms generated by $\gamma$ acting on the \textit{total transversal} $T_\gamma:=\bigsqcup_i T_i$. We denote by $\mathscr{H}_\f$ any element on its equivalence class under differentiable equivalence (in the sense of \cite[\S 2.2]{haefliger2}), when a specific choice of cocycle representing $\f$ is not important. It is clear that the orbit space $T_\f/\mathscr{H}_\f$ coincides with the leaf space $M/\f$.

The notion of fundamental group generalizes to pseudogroups via homotopy classes of $\mathscr{H}$-loops (see \cite[Appendix D]{molino} for details). For the holonomy pseudogroup $\mathscr{H}_\f$, in particular, this furnishes an invariant $\pi_1(\f,\underaccent{\check}{x})$ of $\f$, where $\underaccent{\check}{x}\in T_\f$ is a fixed base point. Similarly to its classical counterpart, its isomorphism class $\pi_1(\f)$ is independent of the base point when $M$ is connected. As described in \cite[Appendix D, Section 1.11]{molino}, there is a natural surjection
\begin{equation}\pi_1(M)\longrightarrow\pi_1(\f).\label{surjection fundamental groups}\end{equation}

A \textit{foliate field} on $M$ is a vector field whose flow preserves $\f$ or, equivalently, a vector field in the Lie subalgebra
$$\mathfrak{L}(\f)=\{X\in\mathfrak{X}(M)\ |\ [X,\mathfrak{X}(\f)]\subset\mathfrak{X}(\f)\}.$$
The quotient $\mathfrak{l}(\f):=\mathfrak{L}(\f)/\mathfrak{X}(\f)$ is the Lie algebra of \textit{transverse vector fields}. One also verifies that each $\overline{X}\in\mathfrak{l}(\f)$ defines a unique section of $\nu\f$ and corresponds to a unique $\mathscr{H}_\f$-invariant vector field on $T_\f$.

A \textit{transverse metric} $\tmetric$ for $\f$ is a symmetric $(2,0)$-tensor field on $M$ such that $\mathcal{L}_X\tmetric=0$ and $\iota_X\tmetric=0$, for all $X\in\mathfrak{X}(\f)$, and $\tmetric(Y,Y)>0$ when $Y\notin T\f$. The triple $(M,\f,\tmetric)$ is then called a \textit{Riemannian foliation}. Notice that in this case, if $(U_i,\pi_i,\gamma_{ij})$ is a cocycle representing $(M,\f)$, then $\tmetric$ projects through each $\pi_i$ to a Riemannian metric on $T_i$, and each $\gamma_{ij}$ becomes an isometry --- or, more concisely, $\mathscr{H}_\gamma$ becomes a pseudogroup of local isometries. Conversely, a Riemannian metric on $T_\gamma$ turning $\mathscr{H}_\gamma$ into a pseudogroup of local isometries coherently pulls, through each $\pi_i$, back to a transverse metric for $\f$.  By \textit{transverse curvature bounds} for $(\f,\tmetric)$ we mean bounds in the corresponding curvature of a transversal $T_\gamma$, with respect to the Levi-Civita connection of its induced metric. In particular, we say that $(\f,\tmetric)$ has sectional curvature bounded above/below by a constant $k$ when the sectional curvature of $T_\gamma$ is bounded above/below by $k$. Notice that this notion is independent of the choice of $T_\gamma$.

A genuine Riemannian metric $\metric$ on $M$ is called \textit{bundle-like} for $\f$ when, for any open set $U$ and any vector fields $Y,Z\in\mathfrak{L}(\f|_U)$ that are perpendicular to the leaves, the function $\metric(Y,Z)$ is \textit{basic}, i.e., constant along the leaves. Any bundle-like metric $\metric$ determines a transverse metric by
\begin{equation}\label{equation bundle-like vs transverse}
    \tmetric(X,Y):=\metric(X^\perp,Y^\perp)
\end{equation}
and, conversely, given $\tmetric$ one can always choose a bundle-like metric that induces it (see \cite[Proposition 3.3]{molino}). A (local) transverse vector field $\overline{X}$ is a \textit{(local) transverse Killing vector field} when $\mathcal{L}_{\overline{X}}\tmetric=0$ (see \cite[Section 3.3]{molino} for details).

We say that a Riemannian foliation $\f$ is \textit{complete} when it admits an associated bundle-like metric which is complete. There is a rich structural theory for complete Riemannian foliations, for which we refer to \cite{molino} or \cite{alex2} for more detailed introductions. Among other properties, it asserts that in this case $\overline{\f}$ is a singular foliation which is furthermore Riemannian with respect to any complete bundle-like metric $\metric$ associated to $\tmetric$, in the sense that each geodesic $\gamma$ that is orthogonal to a leaf closure remains orthogonal to all leaf closures it intersects \cite[Chapter 6]{molino}. A Riemannian metric with this property is said to be \textit{adapted} to $\overline{\f}$. Molino also shows that $\overline{\f}$ is described by the action of a locally constant sheaf $\mathscr{C}_{\mathcal{F}}$ of Lie algebras of germs of transverse Killing vector fields --- now called its \textit{Molino sheaf} ---, in the sense that, for each $x\in M$,
\begin{equation}\label{equation leaf closures described by molino sheaf}T_x\overline{\f(x)}=T_x\f(x)\oplus \{X_x \mid X\in\stalk_x(\mathscr{C}_{\mathcal{F}})\}.\end{equation}
When $M$ is connected, the typical stalk $\mathfrak{g}$ of $\mathscr{C}_{\mathcal{F}}$ is an important algebraic invariant of $\f$, called its \textit{structural algebra}.

A complete Riemannian foliation is a \textit{Killing foliation} when $\mathscr{C}_{\mathcal{F}}$ is globally constant. This class, therefore, includes all complete Riemannian foliations of simply connected manifolds. Another important subclass is that of complete homogeneous Riemannian foliations, that is, those defined by the connected components of the orbits of an isometric Lie group action. In fact, in this case $\mathscr{C}_{\f}$ is determined by the transverse Killing vector fields induced by the action of $\overline{H}<\mathrm{Iso}(M)$ (see \cite[Lemme III]{molino3}). The structural algebra of a Killing foliation is Abelian, and for this reason it is usually denoted by $\mathfrak{a}$. Notice that it is well defined even when $M$ is disconnected. It acts infinitesimally and transversely on $\f$ via the identification $\mathfrak{a}\equiv\mathscr{C}_{\f}(M)$, and in this sense equation \eqref{equation leaf closures described by molino sheaf} becomes $\overline{\f}=\mathfrak{a}\f$ (see also \cite{goertsches} or \cite{caramello2} for more details).

Any Killing foliation $\f$ of a compact manifold $M$ can be deformed into a closed Riemannian foliation $\g$. Our interest in this stems from the fact that, in this case, the leaf space $M/\g$ is naturally an orbifold (see \cite[Proposition 3.7]{molino}) --- a generalization of the concept of manifold that, instead of being locally Euclidean, is locally modeled on quotients of the Euclidean space by finite group actions (see \cite{caramello4} for an introduction). When seen as an orbifold, we will denote the leaf space by $\dslash{M}{\g}$. As pointed out in \cite{haefliger2}, in order to deform $\f$ one can apply \cite[Theorem 3.4]{haefliger2}, that establishes that $\mathscr{H}_\f$ is equivalent to the holonomy pseudogroup of a homogeneous foliation $\f_H$ of a compact orbifold $\mathcal{O}$, given by the orbits of a dense, contractible subgroup $H$ of a torus $\mathbb{T}^N$ that acts smoothly on $\mathcal{O}$. More precisely, the authors show that $(\mathcal{O},\f_H)$ is a realization of the classifying space of $\mathscr{H}_\f$, so that there exists a smooth map $\Upsilon:M\to \mathcal{O}$ such that $\f=\Upsilon^*(\f_H)$. If $\mathfrak{t}$ and $\mathfrak{h}$ are the Lie algebras of $\mathbb{T}^N$ and $H$, respectively, one can choose a path $\mathfrak{h}(t)$ on the Grassmannian $\mathrm{Gr}^{\dim\mathfrak{h}}(\mathfrak{t})$ connecting $\mathfrak{h}$ to a Lie subalgebra $\mathfrak{k}:=\mathfrak{h}(1)<\mathfrak{t}$ whose corresponding Lie subgroup $K$ is closed. Taking $\mathfrak{h}(t)$ in a sufficiently small neighborhood of $\mathfrak{h}$, so that the groups $H(t)$ all act almost freely and $\Upsilon$ remain transverse to all of the induced foliations $\f_{H(t)}$, one obtains a homotopy $\f_t:=\Upsilon^*(\f_{H(t)})$ of $\f$, deforming it into $\g:=\f_1$, which is closed since $K$ is closed.

In \cite{caramello} it is shown that the transverse geometry of $\f$ ``deforms smoothly'' and that some topological properties remain invariant under $\f(t)$. Below we gather what will be useful to us:
\begin{theorem}[{\cite[Theorem B]{caramello}}]\label{theorem deformations}
With the notation established above, the following holds:
\begin{enumerate}[(i)]
\item \label{item deformation respects leaf closures} The deformation occurs within the closures of the leaves of $\f$, in the sense that, for each $x\in M$, the tangent space $T_x\f_t(x)$ is a subspace of $T_x\overline{\f}(x)$, for all $t\in[0,1]$.
%\item $\g$ can be chosen arbitrarily close to $\f$, considering the foliations as sections of the Grassmannian bundle $\mathrm{Gr}^p(TM)$
\item \label{item deformation of transverse metric} The transverse metric $\tmetric$ deforms smoothly to transverse metrics $\iota_t(\tmetric)$ for $\f_t$, in the sense that $\iota_t(\tmetric)$ is a smooth time-dependent tensor field on $M$.
\item In particular, strict bounds on the transverse sectional curvature $\sec_\f$ can be maintained for $\sec_{\f_t}$.
\item \label{item quotient of deformation is orbit space of torus} The quotient orbifold $\dslash{M}{\g}$, endowed with $\iota(\tmetric):=\iota_1(\tmetric)$, admits an isometric action of a torus $\mathbb{T}^d$, where $d=\dim\mathfrak{a}$, such that $M/\overline{\f}\cong(M/\g)/\mathbb{T}^d$.
\end{enumerate}
\end{theorem}

Here we will be primarily interested in using these deformations in order to obtain a sequence $\g_i$ of closed foliations with the aforementioned properties and so that $\g_i\to \f$, in the sense that $T\g_i\to T\f$, as sections of the Grassmannian $\mathrm{Gr}^p(TM)$. By what we saw above, to do so we are led to consider a sequence of Lie subalgebras $\mathfrak{k}_i\subset\mathfrak{t}$ converging to $\mathfrak{h}$, and take $\g_i$ as the corresponding foliations --- one can even interpolate a path trough the points $\mathfrak{h}_i$ to make the sequence part of a deformation, if desired. Following the terminology in \cite{caramello2} for deformations, we will refer to such a sequence $\g_i$ as a \textit{regular} sequence of closed foliations converging to $\f$, or simply say that $\g_i\to \f$ \textit{regularly}. In this case we will denote the respective induced transverse metrics by $\iota_i(\tmetric)$. Notice that, by Theorem \ref{theorem deformations}(\ref{item deformation respects leaf closures}), for any $x\in M$ we have $\g_i(x)\subset \overline{\f}(x)$.

Let us now establish that regular convergence implies that $T_x\g_i(x)\to T_x\f(x)$, uniformly on $x$. For this, recall first that an inner product $\proin{\ }{\ }$ on a vector space $V$ provides an embedding of $\mathrm{Gr}^{\ell}(V)$ into $\operatorname{End}(V)$, by associating $E\in \mathrm{Gr}^{\ell}(V)$ to the orthogonal projection $P_E: V \to V$ onto $E$. Then there is a natural distance function on $\mathrm{Gr}^{\ell}(V)$ defined via this embedding:
\[d_\mathrm{Gr}(E, F) := \|P_E - P_F\|_{\text{op}} = \sup_{\|v\| = 1} \|P_E(v) - P_F(v)\|.\]

\begin{proposition}\label{proposition uniform convergence of foliations}
    If $\g_i\to \f$ regularly, then $T\mathcal{G}_i \to T\mathcal{F}$ uniformly, that is, in the $C^0$ topology of $\Gamma(\mathrm{Gr}^p(M))$.
\end{proposition}

\begin{proof}
    Denote $\ell:=\dim\mathfrak{h}$. Passing to the infinitesimal $\mathbb{T}^N$-action, let $\operatorname{ev}_o\colon \mathfrak{t} \to T_o\mathcal{O}$ denote the evaluation of the fundamental vector fields at $o\in\mathcal{O}$. Since $H$ acts locally freely, each $\operatorname{ev}_o$ will remain injective near $\mathfrak{h}$: the topology of $\mathrm{Gr}^p(T_o\mathcal{O})$ garantees that injectivity is an open condition. Therefore, there is a tube $K=\mathcal{O}\times B$ around $\mathcal{O}\times\{\mathfrak{h}\}$ in $\mathcal{O} \times \mathrm{Gr}^{\ell}(\mathfrak{t})$ where the map
    \begin{align*}
    \operatorname{ev}\colon K  & \longrightarrow \mathrm{Gr}^{\ell}(T\mathcal{O})\\
    (o, \mathfrak{s}) & \longmapsto \operatorname{ev}_{o}(\mathfrak{s})
    \end{align*}
    is well defined. It is also continuous, since $\mathbb{T}^N \begin{tikzcd}[ampersand replacement=\&, cramped, sep=small] \& \arrow[loop left] \phantom{X} \end{tikzcd} \hspace{-8pt} \mathcal{O}$ is smooth, and we can further assume that it is uniformly continuous by passing to a compact $K$ if necessary. By construction, we have $d_\mathrm{Gr}(\mathfrak{k}_i,\mathfrak{h})\to 0$. The uniform continuity of $\operatorname{ev}$ then gives us
    \[\sup_{o\in\mathcal{O}}d_\mathrm{Gr}(\operatorname{ev}_{o}(\mathfrak{k}_i),\operatorname{ev}_{o}(\mathfrak{h}))=\sup_{o\in\mathcal{O}}d_\mathrm{Gr}(T_o\f_{K_i},T_o\f_H)\to 0,\]
    that is, $T\f_{K_i} \to T\mathcal{F}_H$ uniformly.
    
    Now the compactness of $M$ allows us to, similarly, pass this uniform convergence through the pullback. In more detail, let $\sigma \colon M\to \mathrm{Gr}^\ell(\mathcal{O})$ be given by $\sigma(x):=T_{\Upsilon(x)}\f_H$ and, similarly, define $\sigma_i(x):=T_{\Upsilon(x)}\f_{K_i}$. Notice that these can be seen as sections of the pullback bundle $\Upsilon^*\mathrm{Gr}^\ell(\mathcal{O})$, and that the uniform convergence $T\f_{K_i} \to T\mathcal{F}_H$ immediately implies $\sigma_i\to\sigma$ uniformly. Let $\Upsilon^*\mathrm{Gr}^\ell(\mathcal{O})^\pitchfork$ be the subset of $\Upsilon^*\mathrm{Gr}^\ell(\mathcal{O})$ consisting of the $\Upsilon$-transverse pairs $(x,E)$, that is, those verifying
    \[\operatorname{Im}(\dif \Upsilon_x)+E=T_{\Upsilon(x)}\mathcal{O}.\]
    Since $\dif \Upsilon_x$ varies smoothly on $x$ and transversality is an open condition, it follows that $\Upsilon^*\mathrm{Gr}^\ell(\mathcal{O})^\pitchfork$ is open in $\Upsilon^*\mathrm{Gr}^\ell(\mathcal{O})$. It also clearly contains the images $\sigma(M)$ and $\sigma_i(M)$, for all $i$, by construction. Moreover, since $\sigma_i\to \sigma$ uniformly, we can find a compact tubular neighborhood $K'\subset\Upsilon^*\mathrm{Gr}^\ell(\mathcal{O})^\pitchfork$ around $\sigma(M)$ containing all $\sigma_i(M)$, for $i$ suficiently large. Transversality ensures that the map
    \begin{align*}
    \Psi\colon \Upsilon^*\mathrm{Gr}^\ell(\mathcal{O})^\pitchfork & \longrightarrow \mathrm{Gr}^p(M)\\
     (x,E) & \longmapsto (\dif \Upsilon_x)^{-1}(E)
    \end{align*}
     is well defined and continuous. Restricted to $K'$, it is furthermore uniformly continuous. Hence, since $\sup_{x\in M} d_\mathrm{Gr}(\sigma_i(x),\sigma(x))\to 0$, we obtain
     \begin{align*}
\sup_{x\in M} d_\mathrm{Gr}(T_x\g_i,T_x\f) & = \sup_{x\in M} d_\mathrm{Gr}((\dif \Upsilon_x)^{-1}T_{\Upsilon(x)}\f_{K_i},(\dif \Upsilon_x)^{-1}T_{\Upsilon(x)}\f_H)\\
 & = \sup_{x\in M} d_\mathrm{Gr}(\Psi(\sigma_i(x)),\Psi(\sigma(x)))\longrightarrow 0 \qedhere
     \end{align*}
\end{proof}

\color{black}

\section{Transverse diameter and deformations}\label{section: diameter and deformations}

In this section we will prove Theorem \ref{theoremA}. We start by defining and briefly studying the notion of transverse diameter. Let $(M,\f,\tmetric)$ be a Riemannian foliation. We define the \textit{transverse length} of a piecewise-smooth curve $\alpha:[a,b]\to M$ by
\[\ell_\intercal(\alpha):=\int_a^b \sqrt{\tmetric(\alpha'(t),\alpha'(t))}\;dt.\]
If $M$ is connected, for $L,L'\in\f$ we can then consider the \textit{transverse pseudodistance} $d_\intercal(L,L')=\inf\ell_\intercal(\alpha)$, where the infimum runs over the set of all piecewise-smooth curves $\alpha$ starting at $L$ and ending at $L'$. Endowed with $d_\intercal$, the leaf space $M/\f$ becomes a pseudometric space. In fact, clearly $d_\intercal$ is symmetric and satisfies $d_\intercal(L,L)=0$ for any $L\in\f$. The triangle inequality stems from the fact that, for $L_1,L_2,L_3\in\f$, if $\alpha_{21}$ and $\alpha_{13}$ connect $L_2$ to $L_1$ and $L_1$ to $L_3$, respectively, then one can choose $\alpha$ inside $L_1$ such that $\alpha_{21}*\alpha*\alpha_{13}$ connects $L_2$ to $L_3$, leading to the result since $\ell_\intercal(\alpha)=0$. Notice, furthermore, that positivity fails in general, since non-trivial leaf closures contain other leaves.

We define the \textit{transverse diameter} of $\f$ to be $\diam(M/\f,d_\intercal)$, that is,
\[\diam_\intercal(\f)=\sup_{L,L'\in\f}d_\intercal(L,L').\]

Now suppose $(\f,\tmetric)$ is complete and let $\metric$ be any complete bundle-like metric for it. Because $\metric$ is adapted to $\overline{\f}$, the induced Husdorff distance between leaf closures $J_1,J_2\in\overline{\f}$ --- i.e., the infimum of lengths of curves joining them --- is constant, in the sense that, for any $p \in J_1$,
    \[d_\metric(p, J_2) = d_\metric(J_1, J_2)\]
    \cite[Proposition 2.3]{radeschi}. This ensures that $d_\metric$ indeed defines a distance function on $M/\overline{\f}$. In particular, $(M/\f,d_\metric)$ is again a pesudometric space.
    
Endowing $M/\overline{\f}$ with $d_\metric$, the projection $\pi\colon M\to M/\overline{\f}$ becomes a strong submetry \cite[Proposition 3.33]{radeschi}, that is, $\pi(B_r[x])=B_r[\pi(x)]$, for all $x\in M$ and all $r>0$. In particular, one concludes that $M/\overline{\f}$ is boundedly compact, hence complete, by the Hopf--Rinow theorem for length spaces \cite[Theorem 2.5.28 and Remark 2.5.29]{burago} (it is locally compact since $M$ is). In other words, we showed that, if $\f$ is complete, then $M/\overline{\f}$ is a geodesic space.

\begin{proposition}\label{proposition transverse diameter versus diameter of quotient}
    Let $(M,\f,\tmetric)$ be a complete Riemannian foliation of a connected manifold $M$. Then
    \[\diam_\intercal(\f)=\diam(M/\f,d_\metric)=\diam(M/\overline{\f},d_\metric),\]
    regardless of the chosen complete bundle-like metric $\metric$ associated with $\tmetric$.
\end{proposition}

\begin{proof}
First notice that, by equation \eqref{equation bundle-like vs transverse},
\begin{equation}\label{equation inequality transverse length}
    \ell_\intercal(\alpha)\leq \ell_\metric(\alpha)
\end{equation} 
for any piecewise-smooth curve $\alpha:[a,b]\to M$, where $\ell_\metric$ is the usual length induced by $\metric$. Moreover, it is clear that equality holds if, and only if, $\alpha$ is $\metric$-orthogonal to all leaves it intersects.

Since $d_\metric(L,L')=d_\metric(\overline{L},\overline{L'})$, for any given leaves $L,L'\in\f$, we have
\[\diam(M/\f,d_\metric)=\sup_{L,L'\in\f}d_\metric(L,L')=\sup_{J,J'\in\overline{\f}}d_\metric(J,J')=\diam(M/\overline{\f},d_\metric).\]
Moreover, from equation \eqref{equation inequality transverse length}, given any $L,L'\in\f$ we obtain $d_\intercal(L,L')\leq d_\metric(L,L')$, and hence $\diam_\intercal(\f)\leq \diam(M/\f,d_\metric)$.

On the other hand, given $J,J'\in\overline{\f}$, let $\gamma:[a,b]\to M$ be a geodesic realizing the distance between them. It exists since in our situation $M/\overline{\f}$ is a geodesic space, as we saw above; then one can take $\gamma$ to be any horizontal geodesic projecting to a minimizing geodesic between $J$ and $J'$ on $M/\overline{\f}$. Since $\gamma$ is perpendicular to all leaf closures it intersects, the same holds for all leaves of $\f$ it intersects, so
\[d_\metric(J,J')=\ell_\metric(\gamma)=\ell_\intercal(\gamma).\]
We claim that $\ell_\intercal(\gamma)=d_\intercal(\f(\gamma(a)),\f(\gamma(b)))$. In fact, suppose this does not happen. Then there exists a piecewise-smooth curve $\beta:[c,d]\to M$, starting at $\f(\gamma(a))$ and ending at $\f(\gamma(b))$, with $\ell_\intercal(\beta)<\ell_\intercal(\gamma)$. By compactness, we can cover $\beta([c,d])$ with finitely many open sets $U_1,\dots,U_k$ in a cocycle representing $\f$. We can suppose further, without loss of generality, that each corresponding local quotient $T_i$ is geodesically convex with respect to the Riemannian metric induced by $\tmetric$, and that they are indexed so that $U_i\cap U_{i+1}\neq \emptyset$ for $i=1,\dots,k-1$. Choose a point $x_i=\beta(t_i)$ on each of these intersections and define $x_0:=\beta(c)$ and $x_k:=\beta(d)$. For each $i=1,\dots,k$, let $\underaccent{\check}{\gamma}_i$ be the unique minimal geodesic on $T_i$ connecting $\pi_i(x_{i-1})$ to $\pi_i(x_i)$. Since the Riemannian metrics induced on $T_i$ by $\tmetric$ and $\metric$ coincide; we will denote the length of a curve $\underaccent{\check}{\alpha}$ on $T_i$ simply by $\ell(\underaccent{\check}{\alpha})$. Each $\underaccent{\check}{\gamma}_i$ lifts horizontally to a geodesic segment $\gamma_i$ on $U_i$, which is therefore orthogonal to all leaves it intersects and satisfies $\ell_\intercal(\gamma_i)=\ell_\metric(\gamma_i)=\ell(\underaccent{\check}{\gamma}_i)$ (see \cite[Proposition 3.5]{molino} for details). Therefore $\ell_\intercal(\gamma_i)=\ell_\metric(\gamma_i)\geq d_\metric(\f(x_{i-1}),\f(x_i))$. Since also $\ell(\underaccent{\check}{\gamma}_i)\leq \ell(\pi_i\circ \beta|_{[x_{i-1},x_i]})=\ell_\intercal(\beta|_{[x_{i-1},x_i]})$, it follows
\begin{align*}
d_\metric(J,J') & =d_\metric(\f(x_0),\f(x_k))\leq \sum_{i=1}^k d_\metric(\f(x_{i-1}),\f(x_i))\\
 & \leq \sum_{i=1}^k \ell_\intercal(\gamma_i) \leq \sum_{i=1}^k \ell_\intercal(\beta|_{[x_{i-1},x_i]}) = \ell_\intercal(\beta)\\
 & <\ell_\intercal(\gamma)=d_\metric(J,J'),
\end{align*}
a contradiction. We therefore have $\ell_\metric(\gamma)=d_\intercal(\f(\gamma(a)),\f(\gamma(b)))$, so taking the supremum on both sides, over all geodesics minimizing the distance between two leaf closures, we obtain $\diam(M/\overline{\f},d_\metric)\leq \diam_\intercal(\f)$.
\end{proof}

It is worth noticing the following corollary of the above proof.

\begin{corollary}
Let $(M,\f,\tmetric)$ be a closed, complete Riemannian foliation of a connected manifold $M$. Then $(M/\f,d_\intercal)$ and $(M/\f,d_\metric)$ are isometric, for any complete bundle-like metric $\metric$ associated with $\tmetric$.
\end{corollary}

Now let us focus on a non-closed Killing foliation $\f$ of a compact manifold $M$. Endow $\f$ with a bundle like metric $\metric$, and let $\g_i$ be a sequence of closed foliations converging regularly to $\f$. Consider the metric $\metric_i$ obtained by restricting $\metric$ to $T\g_i$ and $T\f^{\perp_\metric}$ and declaring that these distributions are orthogonal --- for that we assume, by ignoring the first terms of the sequence if necessary, that the distributions are always complementary. Notice that, therefore, the distributions $(T\f)^{\perp_\metric}$ and $(T\g_i)^{\perp_{\metric_i}}$ are the same. This way, the tensor field
\[(\metric_i)_\intercal(X,Y):=\metric_i(X^{\perp_{\metric_i}},Y^{\perp_{\metric_i}})=\metric(X^{\perp_{\metric_i}},Y^{\perp_{\metric_i}})\]
is precisely $\iota_i(\tmetric)$ from item (\ref{item deformation of transverse metric}) of Theorem \ref{theorem deformations} (see \cite[p. 11]{caramello} for the general construction). In particular, since $(\metric_i)_\intercal$ is therefore a $\g_i$-transverse Riemannian metric, $\metric_i$ is bundle-like for $\g_i$. It will be instructive to establish the following technical lemma.

\begin{lemma}\label{lemma: C0 convergence of metrics}
With the notation established above, $\metric_i\to \metric$ uniformly.
\end{lemma}

\begin{proof} 
Let $P\colon TM \to T\mathcal{F}$ be the $\metric$-orthogonal projection, and likewise, let $P_i\colon TM \to T\mathcal{G}_i$ be the oblique projection defined by the splitting $TM = T\mathcal{G}_i \oplus T\mathcal{F}^{\perp_\metric}$. Similarly to what we did in the proof of Proposition \ref{proposition uniform convergence of foliations}, let $\mathrm{Gr}^p(TM)^\pitchfork$ be the open subbundle of spaces complementary to $T\mathcal{F}^{\perp_\metric}$. It is not dificult to check that the map $\Psi\colon \mathrm{Gr}^p(TM)^\pitchfork \to \mathrm{End}(TM)$ that associates $E\in\mathrm{Gr}^p(T_xM)^\pitchfork$ to the projection $P_E\colon T_xM\to T_xM$ onto $E$ with respect to the decomposition $TM=E\oplus T_x\mathcal{F}^{\perp_\metric}$ is well defined and continuous. Moreover, we can suppose $\Psi$ is uniformly continuous, after restricting it to a compact tubular neighborhood containing the sections $T\mathcal{F}$ and $T\mathcal{G}_i$, for sufficiently large $i$. Since $T\mathcal{G}_i\to T\mathcal{F}$ uniformly, by Proposition \ref{proposition uniform convergence of foliations}, we thus obtain
\[\delta_i:=\sup_{x\in M} \|P_i-P\|_{\text{op}} = \sup_{x\in M} \|\Psi(T_x\mathcal{G}_i)-\Psi(T_x\mathcal{F})\|_{\text{op}}\longrightarrow 0.\]

Now notice that
\begin{align*}
\metric(v, w) & = \metric(Pv, Pw) + \metric((I - P)v, (I - P)w),\\
\metric_i(v, w) & = \metric(P_i v, P_i w) + \metric((I - P_i)v, (I - P_i)w).
\end{align*}
Hence, defining
\begin{align*}
A(v,w) & := \metric(P_i v, P_i w)-\metric(Pv, Pw),\\
B(v,w) & :=\metric((I - P_i)v, (I - P_i)w)-\metric((I - P)v, (I - P)w),
\end{align*}
we obtain
\begin{equation}\label{equation metrics converge uniformly} \sup_{v,w\in \mathbb{S}^1_\metric M} |\metric_i(v,w)-\metric(v,w)| \leq \sup_{v,w\in \mathbb{S}^1_\metric M} \big(|A(v,w)|+|B(v,w)|\big).\end{equation}
To bound $|A(v,w)|$, we write
\begin{align*}A(v,w) & = \metric(P_i v, P_i w) - \metric(Pv, P_i w) + \metric(Pv, P_i w) - \metric(Pv, Pw)\\
 & = \metric(P_i v- Pv, P_i w)+\metric(Pv, P_i w-Pw)
\end{align*}
and apply the Cauchy--Schwarz inequality:
\begin{align*}
|A(v,w)| & \leq \|P_i v - Pv\|_\metric \|P_i w\|_\metric + \|Pv\|_\metric \|P_i w - Pw\|_\metric\\
 & \leq \|P_i - P\|_{\text{op}} \|v\|_\metric(\|Pw\|_\metric + \|P_i w - Pw\|_\metric)\\
 & \phantom{=:} +\|Pv\|_\metric\|P_i - P\|_{\text{op}} \|w\|_\metric\\
 & \leq \delta_i(1+\delta_i)+\delta_i=2\delta_i+\delta_i^2.
\end{align*}
The term $|B(v,w)|$ is handled analogously. Notice that $(I - P_i) - (I - P) = P - P_i$, so  $\|(I - P_i) - (I - P)\|_{\text{op}} = \|P - P_i\|_{\text{op}} \leq \delta_i$, and that  $(I - P)$ is the $\metric$-orthogonal projection onto $T\mathcal{F}^{\perp_\metric}$, so $\|I - P\|_{\text{op}}=1$. Therefore the same Cauchy--Schwarz argument leads to the identical bound:
\[|B(v,w)|\leq 2\delta_i+\delta_i^2.\]
Thus \eqref{equation metrics converge uniformly} gives
\[\sup_{v,w\in \mathbb{S}^1_\metric M} |\metric_i(v,w)-\metric(v,w)|\leq 4\delta_i+2\delta_i^2 \longrightarrow 0.\qedhere\]
\end{proof}
\color{black}

We will henceforth denote by $\ell_i$ the length of curves induced by $\metric_i$ and by $d_i$ the corresponding induced metric. By abuse, the metric on $M/\g_i$ will also be denoted by $d_i$.

\begin{corollary}\label{corollary: uniform convergence of the metrics}
    With the notation established above, $d_i\to d_\metric$ uniformly, as metrics on $M$. Consequently, the same holds for the induced distances between subsets, i.e.,
    \[\lim_{i\to\infty} |d_i(A, B) - d_\metric(A, B)|=0,\]
    for any $A,B\subset M$.
\end{corollary}

\begin{proof}
    Let $\varepsilon>0$ be given. We claim that, for sufficiently large $i$, the bi-Lipschitz estimate
    \[(1 - \varepsilon)\|\cdot\|_\metric^2 \leq \|\cdot\|_{\metric_i}^2 \leq (1 + \varepsilon)\|\cdot\|_\metric^2\]
    holds. In fact, we know from Lemma \ref{lemma: C0 convergence of metrics} that $\metric_i\to \metric$ uniformly, hence $\|\ \|_{\metric_i}\to \|\ \|_\metric$ uniformly, and thus
\[\delta_i:=\sup_{w \in \mathbb{S}^1_\metric M} \big| \|w\|_{\metric_i}^2 - 1 \big| = \sup_{w \in \mathbb{S}^1_\metric M} \big| \|w\|_{\metric_i} - \|w\|_\metric \big| \longrightarrow 0.\]
    This yields
    \[(1-\delta_i)\|v\|_\metric^2\leq \|v\|_{\metric_i}^2\leq(1+\delta_i)\|v\|_\metric^2\]
    for any $v\in TM$, after writing $v=\|v\|_\metric(v/\|v\|_\metric)$. Therefore
    \[|d_i(x, y) - d_\metric(x, y)| \leq \varepsilon d_\metric(x, y) \leq \varepsilon \diam(M, d_\metric),\]
    for any $x, y \in M$, establishing that $d_i \to d_\metric$ uniformly on $M \times M$. This uniform convergence extends to the distance between subsets:
    \[|d_i(A, B) - d(A, B)| \leq \sup_{\substack{x \in A \\ y\in B}} |d_i(x, y) - d(x, y)| < \varepsilon,\]
    for sufficiently large $i$, independent of the choice of subsets $A$ and $B$.
\end{proof}

Let us now study the metrics induced on $M/\overline{\f}$ by $d_i$ and $d_\metric$.

\begin{proposition}\label{proposition g and gi induce same distance between leaf closures}
    Let $\f$ be a Killing foliation of a connected, compact manifold $M$ and let $\g_i$ be a regular sequence of closed foliations converging to $\f$. Let $\metric$ be a bundle-like metric for $\f$ and $\metric_i$ be the induced bundle-like metric for $\g_i$. Then $\metric_i$ is adapted to $\overline{\f}$ and $(M/\overline{\f},d_i)$ is isometric to $(M/\overline{\f},d_\metric)$, for each $i$.
\end{proposition}

\begin{proof}
By item (\ref{item deformation respects leaf closures}) of Theorem \ref{theorem deformations}, $\g_i$ is a subfoliation of $\overline{\f}$, hence
\[(T\overline{\f})^{\perp_{\metric_i}}\subset(T\g_i)^{\perp_{\metric_i}}=(T\f)^{\perp_\metric}.\]
Since also clearly $(T\overline{\f})^{\perp_{\metric}}\subset (T\f)^{\perp_\metric}$, and on $(T\f)^{\perp_\metric}$ the metrics $\metric$ and $\metric_i$ coincide by definition, it follows that
\[(T\overline{\f})^{\perp_{\metric}}=(T\overline{\f})^{\perp_{\metric_i}}.\]
We thus can henceforth unambiguously mention orthogonality to $\overline{\f}$ without refering to $\metric$ and $\metric_i$. Notice that the lenghts of vectors that are orthogonal to $\overline{\f}$, with respect $\metric$ and $\metric_i$, are also the same.

If a $\metric$-geodesic segment $\gamma$ minimizes the distance between $J,J'\in\overline{\f}$, then it is orthogonal to those leaf closures and hence to $\overline{\f}$. Similarly, the same occurs for a $\metric_i$-geodesic segment $\gamma_i$: if it realizes $d_i(J,J')$, then it is orthogonal to $J$ and $J'$, in particular $\metric_i$-orthogonal to $\g_i(\gamma_i(0))$, hence always $\metric_i$-orthogonal to $\g_i$, thus always $\metric$-orthogonal to $\f$. Therefore
\[d_\metric(J,J')=\ell_\metric(\gamma)=\ell_i(\gamma)\geq d_i(J,J')=\ell_i(\gamma_i)=\ell_\metric(\gamma_i)\geq d_\metric(J,J'),\]
giving us $d_\metric(J,J')=d_i(J,J')$. In particular, if $J'$ is contained within a $\metric$-tubular neighborhood of $J$ \cite[Proposition 16]{mendes} then any $\metric$-geodesic segment emanating orthogonaly from $J$ and ending in $J'$ is minimizing, hence also minimizing with respect to $d_i$. It therefore follows that geodesics with respect to $\metric$ and $\metric_i$ coincide when they are orthogonal to $\overline{\f}$ (since this holds locally), and thus $\metric_i$ is adapted to $\overline{\f}$.
\end{proof}

Since $M/\overline{\f}=(M/\g_i)/\mathbb{T}^d$ and the $\mathbb{T}^d$-action is isometric, it follows from Proposition \ref{proposition g and gi induce same distance between leaf closures} that
\[\diam(M/\overline{\f},d_\metric)=\diam(M/\overline{\f},d_i)\leq \diam(M/\g_i,d_i).\]
Coupling this with Proposition \ref{proposition transverse diameter versus diameter of quotient}, we obtain:

\begin{corollary}\label{corollary diameters of closed approximations}
    Let $\f$ be a Killing foliation of a connected, compact manifold $M$ and let $\g$ be a closed approximation of $\f$ given by a regular deformation. Then $\diam_\intercal (M/\f)\leq \diam_\intercal (M/\g)$, with respect to the (pseudo-)metrics induced by $\metric_\intercal$ and $\iota(\metric_\intercal)$, respectively.
\end{corollary}

\section{Transverse diameter sphere theorem}

Armed with the results from Section \ref{section: diameter and deformations}, we can now proceed to the proof of Theorem \ref{theoremA} of the introduction. We re-state it here for convenience.

\begin{theorem}[Transverse diameter sphere theorem]\label{theorem transverse diameter sphere}
    Let $\f$ be a Killing foliation of codimention $q$ of a connected, compact manifold $M$, satisfying $\sec_\f>1$ and $\diam_\intercal(\f)>\pi/2$, and let $\mathfrak{a}$ be its structural algebra. Then $M$ admits a foliation $\g$, arbitrarily close to $\f$, such that $M/\g$ is homeomorphic to the quotient of $\mathbb{S}^q$ by a finite group $\Gamma<\mathrm{O}(q+1)$. Moreover there is a continuous $\mathbb{T}^d$-action on $\mathbb{S}^q/\Gamma$, where $d=\dim(\mathfrak{a})$, such that 
    \[\frac{M}{\overline{\f}}\cong \frac{\mathbb{S}^q/\Gamma}{\mathbb{T}^d}.\]
\end{theorem}

\begin{proof}
    Let $\g$ be a closed approximation of $\f$ given by a regular deformation. In particular, we can suppose $\sec_{\g}>1$, with respect to $\iota(\tmetric)$. Being transverse, this metric projects down to $\dslash{M}{\g}$, which is therefore a Riemannian orbifold satisfying $\sec_{\dslash{M}{\g}}>1$. Hence $M/\g$ is an Alexandrov space with curvature bounded below by $1$. By Proposition \ref{proposition transverse diameter versus diameter of quotient} and Corollary \ref{corollary diameters of closed approximations} we also have
    \[\diam(M/\g)=\diam_\intercal(\g)\geq \diam_\intercal(\f)>\pi/2.\]
    
    Choose $L$ and $L'$ realizing the diameter of $M/\g$. Then Perelman's generalization of the Grove--Shiohama theorem states that $M/\g$ is homeomorphic to the suspension $\susp(\partial B(L,r))$, for all $r\in(0,\diam(M/\g))$ (see \cite[\S 13.6]{BurGroPer92} or \cite[Theorem 4.5]{perelman}). Taking $r$ small enough so that $B(L,r)$ is contained in an orbifold chart around $L$, we have
    \[\partial B(L,r)=\exp_L(\mathbb{S}^{q-1}(r)/\Gamma),\]
    where $\mathbb{S}^{q-1}(r)$ is the sphere of radius $r$ in $(T_L(\dslash{M}{\g}),\iota(\tmetric)_L)$ and $\Gamma$ is the orbifold local group at $L$, which is isomorphic to the holonomy group of $L$ as a leaf of $\g$. Therefore, lifting the $\Gamma$-action to $\susp \mathbb{S}^{q-1}(r)$ by $h{\mkern 1mu\cdot\mkern 1mu}[v,t]:=[hv,t]$ and linearly identifying $\mathbb{S}^{q-1}(r)$ with the equator of $\mathbb{S}^q$, we get
    \[M/\g\cong \susp\left(\frac{\mathbb{S}^{q-1}(r)}{\Gamma}\right)\cong \frac{\susp \mathbb{S}^{q-1}(r)}{\Gamma}\cong\frac{\mathbb{S}^{q}}{\Gamma}.\]
    Moreover, $M/\overline{\f}\cong (M/\g)/\mathbb{T}^d$, by item (\ref{item quotient of deformation is orbit space of torus}) of Theorem \ref{theorem deformations}.
\end{proof}

We can, alternatively, describe $M/\overline{\f}$ as the quotient of a group action on $\mathbb{S}^q$ by lifting the action of $\mathbb{T}^d$. Let $G<\textnormal{Homeo}(\mathbb{S}^q,\mathbb{S}^q)$ be the subgroup consisting of homeomorphisms projecting to the homeomorphisms of $\mathbb{S}^q/\Gamma$ given by the action of $\mathbb{T}^d$ (see \cite{hafliger3} Section 1.2). As constructed, $G$ is an extension of $\mathbb{T}^d$ by $\Gamma$, that is, we have an exact sequence
    \[1\longrightarrow \Gamma\longrightarrow G\longrightarrow \mathbb{T}^d \longrightarrow 1.\]
Here $\Gamma\to G$ is injective, since the $\Gamma$-action on $\mathbb{S}^q$ is effective. Moreover, $G\to\mathbb{T}^d$, given by the assignment $f\mapsto t$, for $t\in\mathbb{T}^d$ the element whose action coincides with the projection of $f$, is well-defined by the effectiveness of $\mathbb{T}^d$ and surjective by construction. Exactness follows from the fact that $\Gamma$ acts by homeomorphisms that project to the identity. Thus we have
    \[\frac{\mathbb{S}^q}{G}\cong\frac{\mathbb{S}^q/\Gamma}{G/\Gamma}\cong\frac{\mathbb{S}^q/\Gamma}{\mathbb{T}^d},\]
 where the action of $G/\Gamma$ on $\mathbb{S}^q/\Gamma$ is given via the isomorphism $G/\Gamma\cong\mathbb{T}^d$. Thus, we have established the following:

\begin{corollary}\label{corollary of theorem A}
    In the conditions of Theorem \ref{theorem transverse diameter sphere}, there exists an extension $G$ of $\mathbb{T}^d$ by $\Gamma$ such that
    \[\frac{M}{\overline{\f}}\cong \frac{\mathbb{S}^q}{G}.\]
\end{corollary}

\section{Transverse quarter-pinched sphere theorem}

In this final section we will establish our other two main results. We start with the following technical lemma to be used in the proof of Theorem \ref{theoremB}, which in turn will be the main ingredient in the proof Theorem \ref{theoremC}.

\begin{lemma}\label{lemma: diameter on leaf clo}
    Let $\f$ be a Killing foliation of a connected, compact manifold $M$ and let $\g_i\to \f$ be a regular sequence. Let $\metric$ be a bundle-like metric for $\f$ and $\metric_i$ be the induced bundle-like metric for $\g_i$. Then, for any $J\in \overline{\f}$,
    \[\lim_{i\to\infty}\diam(J/\g_i,d_i)=0,\]
    where $d_i$ is the metric induced by $\metric_i$.
\end{lemma}

\begin{proof}
    Because $\g_i \to \f$, the leaves of $\g_i$ inside $J$ increasingly approximate the leaves of $\f$, which are all dense in $J$. Thus, for any given pair of leaves $L_i,L_i'\in \g_i|_J$, we have $d_\metric(L_i,L_i')\to 0$. In particular, given $\varepsilon>0$, there exists $N_1\in\mathbb{N}$ such that $d_\metric(L_i,L_i')<\varepsilon/2$ whenever $i\geq N_1$. Additionally, Corollary \ref{corollary: uniform convergence of the metrics} tells us that there exists $N_2\in\mathbb{N}$ such that $|d_i(L_i,L_i')-d_\metric(L_i,L_i')|<\varepsilon/2$, whenever $i\geq N_2$. Therefore, for all $i\geq \max\{N_1,N_2\}$,
    \[d_i(L_i,L_i')\leq d_\metric(L_i,L_i')+|d_i(L_i,L_i')-d_\metric(L_i,L_i')|<\frac{\varepsilon}{2}+\frac{\varepsilon}{2}=\varepsilon,\]
    thus $\diam(J/\g_i,d_i)\to 0$.
\end{proof}

\begin{theorem}\label{theorem regular convergence implies GH convergence}
    Let $\f$ be a Killing foliation of a connected, compact manifold $M$ and let $\g_i\to \f$ be a regular sequence. Let $\metric$ be a bundle-like metric for $\f$ and $\metric_i$ be the induced bundle-like metric for $\g_i$. Then
    \[(M/\g_i,d_i)\xrightarrow{\textnormal{GH}} (M/\overline{\f},d_\metric).\]
\end{theorem}

\begin{proof}
    Let $\varepsilon>0$ be given. For each $J\in\overline{\f}$, choose a tubular neighborhood $\mathrm{Tub}(J)$ of radius smaller than $\varepsilon/2$ \cite[Proposition 16]{mendes}. This provides an open covering of $M$ from which we can extract, by compactness, a finite sub-cover $\{\mathrm{Tub}(J_1),\dots,\mathrm{Tub}(J_r)\}$. Therefore $\{J_1,\dots, J_r\}$ is an $\varepsilon/2$-dense (hence also $\varepsilon$-dense) set in $M/\overline{\f}$. Since $\g_i\to\f$ regularly, for each of the leaf closures $J_j$, Lemma \ref{lemma: diameter on leaf clo} guarantees that there is $k_j\in\mathbb{N}$ so that $i\geq k_j$ implies $d_i(L_i,L_i')<\varepsilon/2$, for any $L_i,L_i'\in\g_i|_{J_j}$. We have only a finite number $r$ of leaf closures to deal with, so taking $k:=\max_j{k_j}$ we ensure $d_i(L_i,L_i')<\varepsilon/2$, for any $L_i,L_i'\in\g_i$ that are both contained in one of the leaf closures $J_j$, provided $i\geq k$.

    Now choose a leaf $L_s\in\g_i$ inside $J_s$, for each $s=1,\dots,r$. We claim that $\{L_1,\dots,L_r\}$ is an $\varepsilon$-dense set in $M/\g_i$, when $i\geq k$. In fact, given any $\g_i(x)$, we know that there is some $J_{s(x)}$ with $d_\metric(\overline{\f}(x),J_{s(x)})< \varepsilon/2$. A $\metric$-geodesic segment $\gamma$ realizing that distance must be orthogonal to $\overline{\f}(x)$ and $J_{s(x)}$ and hence to all leaves of $\overline{\f}$ that it intersects. As we saw earlier, it is therefore also a $\metric_i$-geodesic segment and $\ell_i(\gamma)=\ell_\metric(\gamma)$. Therefore, if $y\in J_{s(x)}$ is the endpoint of $\gamma$, then
    \[d_i(\g_i(x),\g_i(y))\leq \ell_i(\gamma) = \ell_\metric(\gamma) = d_\metric(\overline{\f}(x),J_{s(x)})< \varepsilon/2.\]
    Since $\g_i(y)\subset J_{s(x)}$, we know that $d_i(\g_i(y),L_{s(x)})< \varepsilon/2$, and hence
    \[d_i(\g_i(x),L_{s(x)})\leq d_i(\g_i(x),\g_i(y))+d_i(\g_i(y),L_{s(x)})<\varepsilon/2+\varepsilon/2=\varepsilon.\]

    From Proposition \ref{proposition g and gi induce same distance between leaf closures} we have $d_\metric(J_s,J_{s'})=d_i(J_s,J_{s'})$. Moreover, we have $d_i(L_s,L_{s'})\geq d_i(J_s,J_{s'})$, since $L_s\subset J_s$ and $L_{s'}\subset J_{s'}$, hence
    \begin{equation}\label{equation theorem regular convergence implies GH convergence}
\begin{aligned}
    |d_i(L_s,L_{s'})-d_\metric(J_s,J_{s'})| & =|d_i(L_s,L_{s'})-d_i(J_s,J_{s'})|\\
     & =d_i(L_s,L_{s'})-d_i(J_s,J_{s'}).
\end{aligned}
    \end{equation}
    If $\zeta$ is a $\metric_i$-geodesic segment from $x_s\in J_s$ to $x_{s'}\in J_{s'}$ realizing $d_i(J_s,J_{s'})$, then
    \[d_i(\g_i(x_s),\g_i(x_{s'}))\leq \ell_i(\zeta)=d_i(J_s,J_{s'})\leq d_i(\g_i(x_s),\g_i(x_{s'})),\]
    so equality holds. Substituting in \eqref{equation theorem regular convergence implies GH convergence} and using the triangle inequality,
    \begin{align*}
        |d_i(L_s,L_{s'})-d(J_s,J_{s'})| & = d_i(L_s,L_{s'})-d_i(\g_i(x_s),\g_i(x_{s'}))\\
        & \leq d_i(L_s,\g_i(x_s))+d_i(\g_i(x_s),\g_i(x_{s'}))\\
        & \phantom{=:} +d_i(\g_i(x_{s'}),L_{s'})-d_i(\g_i(x_s),\g_i(x_{s'}))\\
        & = d_i(L_s,\g_i(x_s))+d_i(\g_i(x_{s'}),L_{s'})\\
        &<\varepsilon/2+\varepsilon/2=\varepsilon.
    \end{align*}

    Summing up, we have built $\varepsilon$-dense sets $\{L_1,\dots, L_r\}$ and $\{J_1,\dots,J_r\}$ in $(M/\g_i,d_i)$ and $(M/\overline{\f},d_\metric)$, respectively, such that 
\[ |d_i(L_s,L_{s'})-d_\metric(J_s,J_{s'})| < \varepsilon\]
for all indices $s,s'\in\{1,\dots,r\}$. It then follows by \cite[Example 56]{petersen} that $d_{GH}(M/\g_i,M/\overline{\f})<3\varepsilon$, and so we conclude that $M/\g_i\xrightarrow{\textnormal{GH}} M/\overline{\f}$.
\end{proof}

We stress that the fact that $M/\overline{\f}$ is a Gromov--Hausdorff limit of orbifolds was already known from \cite[Corollary 1.6]{desmarcos} (applied to $\overline{\f}$), in which the orbifolds are obtained via blow-ups. The novelty in Theorem \ref{theorem regular convergence implies GH convergence} is that, in the particular case of Killing foliations, the orbifolds can be taken as $M/\g_i$, which allows much more control of their geometry. We can now, for instance, establish Theorem \ref{theoremC} from the introduction.

\begin{theorem}[Transverse quarter-pinched sphere theorem]\label{teoprincipal}
    Let $\f$ be a $q$-codimensional complete Riemannian foliation of a connected manifold $M$, such that $1/4<\sec_\f< 1$ and $q\geq3$. Then $\f$ develops to a simple foliation $\hat{\f}$ of the universal covering $\hat{M}$, given by the fibers of a submersion onto $\mathbb{S}^q$.
\end{theorem}

\begin{proof}
    Let $p:\hat{M}\to M$ be the universal covering of $M$ and $\hat{\f}:=p^*(\f)$ the corresponding lift of $\f$, endowed with the pullback transverse metric $\hat{\metric}_\intercal:=p^*(\tmetric)$. Then $\hat{\f}$ is a Killing foliation, since $\hat{M}$ is simply connected and the pullback of any complete bundle-like metric for $\f$ associated to $\tmetric$ is complete and bundle-like for $\hat{\f}$. Moreover, it is clear by construction that $1/4<\sec_{\hat{\f}}<1$. By the transverse Bonnet--Myers theorem \cite[Theorem 1]{hebda}, it follows that $\hat{M}/\hat{\f}$ is compact. Therefore $\hat{M}/\overline{\hat{\f}}$ is also compact (see \cite[Section 4.2]{caramello3}).

    We now consider the holonomy pseudogroup $\mathscr{H}_{\hat{\f}}$ of $\hat{\f}$, acting on a total transversal $T$. As mentioned in Section \ref{section preliminaries}, there is a surjective homomorphism
    \[\pi_1(\hat{M})\longrightarrow\pi_1(\mathscr{H}_{\hat{\f}}),\]
    hence $\mathscr{H}_{\hat{\f}}$ is simply connected. Moreover, $T/\overline{\mathscr{H}_{\hat{\f}}}\cong\hat{M}/\overline{\hat{\f}}$ is compact (see also \cite[Proposition 2.3]{molino}), therefore we can apply \cite[Theorem 3.7]{haefliger2} to obtain a Killing foliation $\f'$, of a \emph{compact}, simply connected manifold $M'$, such that $\mathscr{H}_{\hat{\f}}$ is differentiably equivalent to $\mathscr{H}_{\f'}$. Via this transverse equivalence, $\hat{\metric}_\intercal$ induces a transverse metric $\tmetric'$ for $\f'$, with $1/4<\sec_{\f'}< 1$.

    Let $\g_i'\to\f'$ regularly. Then, by Theorem \ref{theorem regular convergence implies GH convergence},
    \begin{equation}
    M'/\g_i'\xrightarrow{\textnormal{GH}}M'/\overline{\f'}.
    \label{equation: proof transverse quarter pinched GH convergence}
    \end{equation}
    Moreover, each $\dslash{M'}{\g_i'}$ is a connected, compact Riemannian orbifold satisfying $1/4<\sec_{\dslash{M'}{\g_i'}}<1$ and $\dim(\dslash{M'}{\g_i'})\geq 3$. By \cite[Proposition 5.2]{bohm} and \cite{brendle} it follows that $\dslash{M'}{\g_i'}$ is orbifold-diffeomorphic to the quotient of $\mathbb{S}^q$ by a finite group $G_i$, as we already discussed in the introduction. Furthermore, as $M'$ is simply connected,
    \[0=\pi_1(\mathscr{H}_{\g_i'})\cong\pi_1^{\mathrm{orb}}(\dslash{M'}{\g_i'})\cong G_i,\]
    therefore each $\dslash{M'}{\g_i'}$ is actually diffeomorphic to $\mathbb{S}^q$.

    We now argue that there is no collapse in the convergence \eqref{equation: proof transverse quarter pinched GH convergence}. In fact, notice that $\dslash{M'}{\g_i'}$ is a simply connected, complete Riemannian manifold satisfying $1/4<\sec_{\dslash{M'}{\g_i'}}<1$, for each $i$. By \cite[Theorem 6]{cheeger} it then follows that its injectivity radius verifies $\text{inj}(\dslash{M'}{\g_i'})\geq\pi$. Moreover, by Günther's volume comparison \cite[Theorem III.4.2]{chavel}, for any fixed $L_i\in\g_i'$,
    \[\vol(\dslash{M'}{\g_i'})\geq\vol(B_i(L_i,\pi))>\vol(B^{\mathbb{S}^q}(\pi)),\]
    where $B^{\mathbb{S}^q}(\pi)$ is a geodesic ball of radius $\pi$ on the round unit sphere. It then follows by \cite[Corollary 10.10.11]{burago} that there is no collapse, that is, the dimension of $M'/\overline{\f'}$, as an Alexandrov space, is $q$. This dimension coincides with the codimension of the restriction of $\overline{\f'}$ to its regular stratum (see, e.g., \cite[Section 1]{grove2}); that is, a leaf closure $\overline{L'}$ of maximal dimension has codimension $q=\codim (L')$. In particular, the structural algebra of $\f'$ is trivial, hence $\f'$ is a closed foliation. Therefore $\dslash{M'}{\f'}$ is a connected, compact Riemannian orbifold with $1/4<\sec_{\dslash{M'}{\f'}}<1$ and $\dim(\dslash{M'}{\f'})\geq 3$, so, as before, we obtain a diffeomorphism between $\dslash{M'}{\f'}$ and $\mathbb{S}^q$. The differentiable equivalence between $\mathscr{H}_{\hat{\f}}$ and $\mathscr{H}_{\f'}$ induces an orbifold diffeomorphism between $\dslash{M'}{\f'}$ and $\dslash{\hat{M}}{\hat{\f}}$, thus giving us the result.
\end{proof}

\subsubsection*{Acknowledgements} We are grateful to Prof. Dirk Töben, Prof. Ivan Pontual and Dr. Henrique Martins for the insightful suggestions and helpful discussions. This study was financed in part by the Coordenação de Aperfeiçoamento de Pessoal de Nível Superior – Brasil (CAPES) – Finance Code 001.

\end{document}